\documentclass[11pt]{amsart}

\usepackage[T1]{fontenc}
\usepackage[utf8]{inputenc}
\usepackage{amsmath,amssymb,amsthm,mathtools}
\usepackage{enumitem}
\usepackage[hidelinks]{hyperref}
\usepackage[margin=1.12in]{geometry}
\usepackage{microtype}
\usepackage{xcolor} 

\newtheorem{theorem}{Theorem}[section]
\newtheorem{proposition}[theorem]{Proposition}
\newtheorem{lemma}[theorem]{Lemma}
\newtheorem{corollary}[theorem]{Corollary}

\newcommand{\R}{\mathbb R}
\newcommand{\Z}{\mathbb Z}
\newcommand{\eps}{\varepsilon}
\newcommand{\Bad}{\operatorname{Bad}}
\newcommand{\diam}{\operatorname{diam}}
\newcommand{\dimH}{\dim_{\mathrm H}}
\newcommand{\cP}{\mathcal P}

\title[Uniform Littlewood counterexamples of positive dimension]{The uniform Littlewood conjecture fails on a set of positive Hausdorff dimension}

\author[Nikita Shulga]{Nikita Shulga}
\address{Nikita Shulga, Sydney Mathematical Research Institute, The University of Sydney, NSW, Australia}
\email{nikita.shulga@sydney.edu.au}

\begin{document}

\begin{abstract}
The uniform Littlewood conjecture (ULC), introduced by Bandi,
Fregoli and Kleinbock, asserts in the two-number case that
$$
 \lim_{Q\to\infty} Q\min_{1\le n\le Q}\|n\xi\|\,\|n\zeta\|=0
$$
for all real $\xi,\zeta$.  It is proven to hold for almost every pair $(\xi,\zeta)$.
Schleischitz, however, has recently disproved the full statement and showed that the set of
counterexamples contains a dense $G_\delta$ set.

We prove that a set of counterexample pairs with the first coordinate being a badly approximable number has Hausdorff dimension at least $3/2$.  We further
show that the set of badly approximable numbers $\xi$ for which there exists
$\zeta$ such that $(\xi,\zeta)$ is a counterexample to ULC has full Hausdorff
dimension.  This contrasts with the classical Littlewood conjecture, for which
the set of possible counterexamples is known to have Hausdorff dimension $0$.
\end{abstract}

\maketitle



\bigskip

\section{Introduction}

Write $\|x\|$ for the distance from $x\in\R$ to the nearest integer.  The
classical Littlewood conjecture asserts that, for every pair of real numbers
$\xi,\zeta$,
$$
 \liminf_{n\to\infty} n\|n\xi\|\,\|n\zeta\|=0.
$$
Bandi, Fregoli and Kleinbock recently considered the corresponding uniform
form, in which one asks whether the best denominator $1\le n\le Q$ always makes
$$
 Q\min_{1\le n\le Q}\|n\xi\|\,\|n\zeta\|
$$
tend to zero as $Q\to\infty$ \cite{BandiFregoliKleinbock}.  Thus the
uniform Littlewood conjecture in dimension two is the assertion that
\begin{equation}\label{eq:ULCzero}
 \lim_{Q\to\infty} Q\min_{1\le n\le Q}\|n\xi\|\,\|n\zeta\|=0
\end{equation}
for all real $\xi,\zeta$.

Schleischitz disproved this conjecture by constructing pairs for which the
corresponding limsup is positive \cite{SchleischitzULC}.  In fact, his theorem
shows that the set of such pairs is large in the topological sense.  However, no bound on the Hausdorff dimension of the set of counterexamples was provided. His result also
left open whether failure can occur when one coordinate, or both coordinates,
have bounded partial quotients; this is part of
\cite[Problem~2(i)]{SchleischitzULC}.  A subsequent note by Moshchevitin gave
a short elementary proof that ULC fails \cite{MoshchevitinULC}, but it does
not address the size of the counterexample set or the existence
of a counterexample with a badly approximable coordinate. Some results on the inhomogeneous ULC conjecture are given in \cite{SchleischitzULCing}.

Let $\mathcal U$ denote the set of counterexamples to the uniform Littlewood
conjecture:
$$
 \mathcal U=\left\{(\xi,\zeta)\in\R^2:
 \limsup_{Q\to\infty} Q\min_{1\le n\le Q}\|n\xi\|\,\|n\zeta\|>0\right\}.
$$
Let $\Bad_1$ denote the set of badly approximable real numbers.  For an
irrational real number $\xi$, one has $\xi\in\Bad_1$ if and only if the
continued-fraction partial quotients of $\xi$ are bounded.  In this notation,
Schleischitz asks whether $\mathcal U\cap(\Bad_1\times\Bad_1)$, or at least
$\mathcal U\cap(\Bad_1\times\R)$, is nonempty.

For $A\ge2$, let $F_A$ denote the set of real numbers in $(0,1)$ whose canonical
continued-fraction expansion is finite or infinite and has all partial quotients
at most $A$.  Its irrational part is
$$
 F_A^\infty=\{[0;a_1,a_2,\ldots]:1\le a_j\le A\text{ for all }j\}.
$$
We put $\delta_A=\dimH F_A^\infty$.  Adding the finite bounded expansions does
not change the Hausdorff dimension.  Hensley's counting theorem for bounded
continuants implies that $\delta_A\to1$ as $A\to\infty$ and supplies the
quantitative estimates used below \cite{Hensley1990,Hensley1992}.

Our main metric result is the following lower bound.

\begin{theorem}\label{thm:pair-dim}
If $A\ge2$ satisfies $\delta_A>(\sqrt{21}-1)/4$, then
$$
 \dimH\bigl(\mathcal U\cap(F_A\times\R)\bigr)
 \ge 2\delta_A-\frac12.
$$
In particular,
$$
 \dimH\bigl(\mathcal U\cap(\Bad_1\times\R)\bigr)\ge \frac32.
$$
\end{theorem}

In particular, Theorem~\ref{thm:pair-dim} settles the ``at least one coordinate
in $\Bad_1$'' alternative of \cite[Problem~2(i)]{SchleischitzULC}.

\begin{corollary}\label{cor:main}
There exist real numbers $\xi,\zeta$ such that $\xi\in\Bad_1$ and
\begin{equation}\label{eq:mainlimsup}
 \limsup_{Q\to\infty} Q\min_{1\le n\le Q}\|n\xi\|\,\|n\zeta\|>0.
\end{equation}
Equivalently,
$$
 \mathcal U\cap(\Bad_1\times\R)\ne\varnothing.
$$
\end{corollary}

We separately prove that, in fact, the set of badly approximable first coordinates $\xi$, for which we can find $\zeta$ such that $(\xi,\zeta)$ is a counterexample to ULC is of full dimension (that is, of dimension $1$).

Define
$$
 \cP_A=\{\xi\in F_A:\exists\zeta\in\R\text{ with }(\xi,\zeta)\in \mathcal U\},
$$
and
$$
 \cP=\{\xi\in\Bad_1:\exists\zeta\in\R\text{ with }(\xi,\zeta)\in \mathcal U\}.
$$

\begin{theorem}\label{thm:first-dim}
For every $A\ge2$ with $\delta_A>3/4$,
$$
 \dimH\cP_A=\delta_A.
$$
Consequently $\dimH\cP=1$.  More precisely, for every nonempty bounded open
interval $U\subset\R$,
$$
 \dimH(\cP\cap U)=1.
$$
\end{theorem}

The distinction from the classical Littlewood conjecture is substantial.  The
set of counterexamples to the classical Littlewood conjecture, if nonempty,
has Hausdorff dimension zero by a theorem of Einsiedler, Katok and
Lindenstrauss \cite{EinsiedlerKatokLindenstrauss}.  By contrast, the present
paper constructs positive-dimensional families of counterexamples to the
uniform statement, even with one badly approximable coordinate.

The interaction between Littlewood's conjecture and badly approximable
coordinates has a substantial literature.  Pollington and Velani proved that,
for every fixed $\xi\in\Bad_1$, there is a set $G(\xi)\subset\Bad_1$ of
Hausdorff dimension one such that
$$
 q\|q\xi\|\,\|q\zeta\|\le \frac{1}{\log q}
$$
for infinitely many $q$, for every $\zeta\in G(\xi)$
\cite{PollingtonVelaniLittlewood}.  In a related direction, Badziahin,
Pollington and Velani proved that finite intersections of weighted badly
approximable sets have full Hausdorff dimension
\cite{BadziahinPollingtonVelani}.  These results concern the classical and
weighted settings rather than the uniform problem studied here.

We also note that Schleischitz's construction draws in part on Bourgain and Kontorovich's work
toward Zaremba's conjecture \cite{BourgainKontorovich}.  By contrast, the
present argument uses Hensley's counting theorem for a fixed continued-fraction
alphabet and proves the required product-set estimate for arbitrary composite
moduli.

\section{Continued fractions with bounded partial quotients}\label{sec:continued-fractions}

For nonnegative quantities $X$ and $Y$, the notation $X\ll_A Y$ means
that $X\le C_A Y$ for some constant $C_A$ depending only on $A$, whereas $X\asymp_A Y$ means
that both $X\ll_A Y$ and $Y\ll_A X$ hold.  We write $|I|$ for the length of an interval and
$B(x,r)$ for the ball of radius $r$.

Every rational in $(0,1)$ has two finite continued-fraction expansions.  We
always use the canonical expansion, whose last digit is at least $2$.
Finite words of partial quotients are denoted in boldface.  Thus, for a word
$\mathbf{w}=(a_1,\ldots,a_s)$, define its convergents $p_j/q_j$ by
$$
 p_{-1}=1,\quad p_0=0,\qquad
 q_{-1}=0,\quad q_0=1,
$$
$$
 p_j=a_jp_{j-1}+p_{j-2},\qquad
 q_j=a_jq_{j-1}+q_{j-2}\qquad(1\le j\le s).
$$
Thus $[0;a_1,\ldots,a_j]=p_j/q_j$, and
$|p_jq_{j-1}-p_{j-1}q_j|=1$.  We write $q(\mathbf{w})=q_s$.  The inverse branch
associated to $\mathbf{w}$ is
\begin{equation}\label{eq:inverse-branch}
 \phi_{\mathbf{w}}(x)=\frac{p_s+p_{s-1}x}{q_s+q_{s-1}x}\qquad(0\le x\le1).
\end{equation}
The full continued-fraction interval with prefix $\mathbf{w}$ is
$\phi_{\mathbf{w}}([0,1])$.  If $\mathbf{u}$ and $\mathbf{v}$ are finite words, write $\mathbf{u}\mathbf{v}$ for their concatenation.  Then
\begin{equation}\label{eq:continuant-concatenation}
 q(\mathbf{u})q(\mathbf{v})\le q(\mathbf{u}\mathbf{v})\le 2q(\mathbf{u})q(\mathbf{v}).
\end{equation}

Fix $A\ge2$.  A word is called admissible if all its digits lie in
$\{1,\ldots,A\}$.  Let
$$
 F_A^\infty=\{[0;c_1,c_2,\ldots]:1\le c_j\le A\text{ for all }j\}
$$
and, taking the union with rationals with bounded partial quotients, let
$$
 F_A=F_A^\infty\cup
 \{[0;b_1,\ldots,b_s]:1\le b_j\le A,\ b_s\ge2\}.
$$
For an admissible word $\mathbf{w}$, let
$$
 F_A(\mathbf{w})=\{x\in F_A:\text{the canonical expansion of $x$ begins with }\mathbf{w}\}.
$$
The sets $F_A$ and $F_A(\mathbf{w})$ are known to be compact.  Also, in the notations above the empty word is allowed, with
$q(\varnothing)=1$ and $F_A(\varnothing)=F_A$.  If $p/q\in F_A$ is rational, its canonical word is
denoted by $\mathbf{w}(p/q)$.

For $q\ge1$, put
$$
 M_A(q)=\{1\le a<q:(a,q)=1,\ a/q\in F_A\},
 \qquad
 M_A(q;I)=\{a\in M_A(q):a/q\in I\}.
$$
Thus $a\in M_A(q;F_A(\mathbf{u}))$ precisely when the canonical word of $a/q$
begins with $\mathbf{u}$.

We use the following immediate corollary of Hensley's classical counting result.

\begin{theorem}[Hensley, \cite{Hensley1990}]\label{thm:hensley}
Let $A\ge2$ and put $\delta_A=\dimH F_A^\infty$.  Then
\begin{equation}\label{eq:hensley-dyadic-global}
 \sum_{T\le q<2T}|M_A(q)|\asymp_A T^{2\delta_A}
 \,\text{ as }T\to\infty.
\end{equation}
\end{theorem}

Hensley proves the asymptotic
$$
 \sum_{q\le X}|M_A(q)|\sim \kappa_A X^{2\delta_A}
 \,\text{ as }X\to\infty
$$
for a constant $\kappa_A>0$. Subtracting the asymptotics at $2T$ and $T$
gives the corresponding estimate for $T<q\le2T$.  Replacing this interval by
$T\le q<2T$ changes the count by $O(T)$, which is negligible because
$\delta_A\ge\delta_2>1/2$, see \cite{Hensley1989}.  This proves
\eqref{eq:hensley-dyadic-global}.

Let $N_r(E)$ be the least number of intervals of length $r$ needed to cover $E$. From Theorem~\ref{thm:hensley} we can deduce the following covering estimate. It is standard, but we include the proof nevertheless.

\begin{proposition}\label{prop:FA-global-cover}
Uniformly for $0<r\le1$,
\begin{equation}\label{eq:FA-global-cover}
 N_r(F_A^\infty)\ll_A r^{-\delta_A}.
\end{equation}
\end{proposition}

\begin{proof}
For $r$ bounded away from zero the assertion is immediate, so assume that
$r$ is sufficiently small and put $R=r^{-1/2}$.  Let $\mathcal S_r$ be the
set of admissible words $\mathbf{w}$ such that
\begin{equation}\label{eq:stopping-denominator}
 q(\mathbf{w}^{-})<R\le q(\mathbf{w}),
\end{equation}
where $\mathbf{w}^{-}$ is obtained by deleting the last digit.  Every infinite
bounded-digit expansion has a unique prefix in $\mathcal S_r$.  Moreover, the
continuant recurrence gives
\begin{equation}\label{eq:stopping-denominator-size}
 R\le q(\mathbf{w})< (A+1)R
 \qquad(\mathbf{w}\in\mathcal S_r).
\end{equation}

The full continued-fraction interval $I(\mathbf{w})=\phi_{\mathbf{w}}([0,1])$ has length
\begin{equation}\label{eq:full-cylinder-length}
 |I(\mathbf{w})|=\frac1{q(\mathbf{w})(q(\mathbf{w})+q(\mathbf{w}^{-}))}\le q(\mathbf{w})^{-2}\le r.
\end{equation}
Since the intervals $I(\mathbf{w})$, $\mathbf{w}\in\mathcal S_r$, cover $F_A^\infty$, it remains
to bound the number of elements in $\mathcal S_r$.

Append the digit $2$ to each $\mathbf{w}\in\mathcal S_r$.  The resulting word $\mathbf{w}(2)$ is
canonical, different words give different rationals, and
$$
 2R\le q(\mathbf{w}(2))=2q(\mathbf{w})+q(\mathbf{w}^{-})<(2A+3)R.
$$
The interval $[2R,(2A+3)R)$ is covered by $O_A(1)$ dyadic intervals.
The upper bound in Theorem~\ref{thm:hensley}, with the implicit constant
enlarged to include bounded values of $R$, therefore gives
$$
 \#\mathcal S_r\ll_A R^{2\delta_A}=r^{-\delta_A}.
$$
This bound together with \eqref{eq:full-cylinder-length} proves
\eqref{eq:FA-global-cover}.
\end{proof}

The preceding proposition has the following local corollary, whose proof we include.

\begin{lemma}\label{lem:FA-regularity}
Uniformly for every bounded interval $W\subset\R$ and every $0<r\le1$,
\begin{equation}\label{eq:FA-cover}
 N_r(F_A^\infty\cap W)
 \ll_A \left(1+\frac{|W|}{r}\right)^{\delta_A}.
\end{equation}
\end{lemma}

\begin{proof}
Write $L=|W|$.  If $L\le r$, then one interval of length $r$ covers the intersection.  If
$L\ge1$, Proposition~\ref{prop:FA-global-cover} gives
$$
 N_r(F_A^\infty\cap W)\le N_r(F_A^\infty)
 \ll_A r^{-\delta_A}
 \le (L/r)^{\delta_A}.
$$
It remains to consider $r<L<1$.

Define $\mathcal S_L$ to be the set of admissible words $\mathbf w$ such that
$$
 q(\mathbf w^- )<L^{-1/2}\le q(\mathbf w).
$$
Then
$$
 L^{-1/2}\le q(\mathbf{w})<(A+1)L^{-1/2}
 \qquad(\mathbf{w}\in\mathcal S_L).
$$
The intervals $I(\mathbf{w})$, $\mathbf{w}\in\mathcal S_L$, have disjoint interiors, and
\eqref{eq:full-cylinder-length} gives
\begin{equation}\label{eq:stopping-interval-size}
 \frac{L}{2(A+1)^2}\le |I(\mathbf{w})|\le L.
\end{equation}
Every interval $I(\mathbf{w})$ that meets $W$ is contained in the $L$-neighbourhood of
$W$.  Since their interiors are disjoint, \eqref{eq:stopping-interval-size}
shows that only $O_A(1)$ members of $\mathcal S_L$ can meet $W$.

For each such $\mathbf{w}$, one has
$F_A^\infty\cap F_A(\mathbf{w})=\phi_{\mathbf{w}}(F_A^\infty)$.  The determinant identity for convergents
and \eqref{eq:inverse-branch} give
$$
 |\phi_{\mathbf{w}}'(x)|=(q(\mathbf{w})+q(\mathbf{w}^{-})x)^{-2}\le q(\mathbf{w})^{-2}\le L
 \qquad(0\le x\le1).
$$
Thus the image under $\phi_{\mathbf{w}}$ of an interval of length $r/L$ has length at
most $r$.  Proposition~\ref{prop:FA-global-cover} therefore implies
$$
 N_r(F_A^\infty\cap F_A(\mathbf{w}))
 \le N_{r/L}(F_A^\infty)
 \ll_A (L/r)^{\delta_A}.
$$
Summing over the $O_A(1)$ intervals that meet $W$ proves
\eqref{eq:FA-cover}.
\end{proof}

We shall also use the following two elementary continued-fraction facts, which
can be deduced from Khinchin's book \cite{KhinchinCF}.  Let $p_j/q_j$ be the
convergents of a finite or infinite continued fraction of $\xi$, and let $n$ be
a positive integer smaller than the final denominator in the finite case.  If
$j$ is the largest index for which $q_j\le n$, then $q_j\le n<q_{j+1}$ and
\begin{equation}\label{eq:best-approx-facts}
 \|n\xi\|\ge\|q_j\xi\|,
 \qquad
 \|q_j\xi\|\ge\frac1{q_j+q_{j+1}}.
\end{equation}

\begin{lemma}\label{lem:finite-badness}
For every $A\ge2$, uniformly in $a\in M_A(q)$,
\begin{equation}\label{eq:finitebad}
 n\left\|\frac{na}{q}\right\|\ge\frac1{A+2}
 \qquad(1\le n<q).
\end{equation}
\end{lemma}

\begin{proof}
Let $p_j/q_j$ be the convergents of $a/q$, and let $j$ be the largest
index for which $q_j\le n$.  Then $q_j\le n<q_{j+1}$.  Since all partial quotients of $a/q$ are at most $A$, we have
$q_{j+1}\le(A+1)q_j$.  Hence \eqref{eq:best-approx-facts} gives
$$
 n\left\|\frac{na}{q}\right\|
 \ge q_j\left\|q_j\frac aq\right\|
 \ge\frac{q_j}{q_j+q_{j+1}}
 \ge\frac1{A+2}.
$$
\end{proof}

\begin{lemma}\label{lem:cylapprox}
For every $A\ge2$ and every admissible word $\mathbf{w}$,
\begin{equation}\label{eq:cyldiam}
 \diam F_A(\mathbf{w})\asymp_A q(\mathbf{w})^{-2}.
\end{equation}
If $\mathbf{w}=\mathbf{w}(p/q)$ is the canonical word of $p/q\in F_A$, then
\begin{equation}\label{eq:cylapprox}
 \left|\xi-\frac pq\right|\ll_A q^{-2}
 \qquad(\xi\in F_A(\mathbf{w})).
\end{equation}
Finally, if $\mathbf{w}$ begins with $\mathbf{u}$, then $F_A(\mathbf{w})\subseteq F_A(\mathbf{u})$.
\end{lemma}

\begin{proof}
By the determinant identity for convergents, \eqref{eq:inverse-branch} satisfies
$$
 |\phi_{\mathbf{w}}'(x)|=(q_s+q_{s-1}x)^{-2}\asymp q(\mathbf{w})^{-2}
 \qquad(0\le x\le1).
$$
The upper bound in \eqref{eq:cyldiam} follows because $F_A(\mathbf{w})$ lies in the
full interval $\phi_{\mathbf{w}}([0,1])$.  For the lower bound, fix two distinct points
$x_0,x_1\in F_A^\infty$ and apply the mean value theorem to
$\phi_{\mathbf{w}}(x_0)-\phi_{\mathbf{w}}(x_1)$.  If $\mathbf{w}=\mathbf{w}(p/q)$, then $p/q=\phi_{\mathbf{w}}(0)$, so the same
derivative estimate proves \eqref{eq:cylapprox}.  Finally, if $\mathbf{w}=\mathbf{u}\mathbf{v}$, then every expansion beginning with $\mathbf{w}$ also begins
with $\mathbf{u}$, which proves the nesting assertion.
\end{proof}

For the rest of the paper, fix $K_A\ge1$ such that
\begin{equation}\label{eq:KA-choice}
 \diam F_A(\mathbf{w})\le K_Aq(\mathbf{w})^{-2},
 \qquad
 \left|\xi-\frac pq\right|\le K_Aq^{-2}
\end{equation}
whenever $\mathbf{w}$ is admissible, $p/q\in F_A$ has word $\mathbf{w}$, and
$\xi\in F_A(\mathbf{w})$.

\begin{lemma}\label{lem:local-rational-upper}
Let $A\ge2$, let $\mathbf{u}$ be an admissible word, and let $U\subset\R$ be a bounded
interval.  Uniformly in $\mathbf{u}$, $U$, and $T\ge2$,
\begin{equation}\label{eq:hensley-local-upper}
 \sum_{T\le q<6T}|M_A(q;F_A(\mathbf{u})\cap U)|
 \ll_A (1+|U|T^2)^{\delta_A}.
\end{equation}
\end{lemma}

\begin{proof}
Distinct reduced fractions $p/q$ and $p'/q'$ with $q,q'<6T$ satisfy
$$
 \left|\frac pq-\frac{p'}{q'}\right|
 =\frac{|pq'-p'q|}{qq'}\ge\frac1{36T^2}.
$$
For every counted $p/q$, choose a point
$x_{p/q}\in F_A^\infty\cap F_A(\mathbf{w}(p/q))$.  By
\eqref{eq:KA-choice}, $|x_{p/q}-p/q|\ll_A T^{-2}$.  Thus the chosen points lie
in an $O_A(T^{-2})$-enlargement of $U$.  The separation argument above also
shows that every interval of length comparable to $T^{-2}$ contains only
$O_A(1)$ chosen points.  Covering the enlarged interval by
Lemma~\ref{lem:FA-regularity}, with $r\asymp_A T^{-2}$, proves
\eqref{eq:hensley-local-upper}.
\end{proof}

\begin{lemma}\label{lem:fixed-prefix-count}
Let $A\ge2$ and let $\mathbf{u}$ be an admissible word.  There is
$T_0=T_0(\mathbf{u})$ such that, for every $T\ge T_0$,
\begin{equation}\label{eq:fixed-prefix-count}
 \sum_{T\le q<6T}|M_A(q;F_A(\mathbf{u}))|
 \asymp_A q(\mathbf{u})^{-2\delta_A}T^{2\delta_A}
 \asymp_A (\diam F_A(\mathbf{u}))^{\delta_A}T^{2\delta_A}.
\end{equation}
The implicit constants are independent of $\mathbf{u}$.
\end{lemma}

\begin{proof}
For the upper bound, every fraction counted on the left lies in the full
continued-fraction interval $\phi_{\mathbf{u}}([0,1])$, whose length is
$\asymp q(\mathbf{u})^{-2}$.  Lemma~\ref{lem:local-rational-upper}, applied to this
interval, gives
$$
 \sum_{T\le q<6T}|M_A(q;F_A(\mathbf{u}))|
 \ll_A \bigl(1+q(\mathbf{u})^{-2}T^2\bigr)^{\delta_A}
 \ll_A q(\mathbf{u})^{-2\delta_A}T^{2\delta_A}
$$
for $T$ sufficiently large relative to $q(\mathbf{u})$.

For the lower bound, put $X=T/q(\mathbf{u})$.  By
Theorem~\ref{thm:hensley}, there are $\gg_A X^{2\delta_A}$ canonical
bounded-digit words $\mathbf{v}$ with
$$
 \frac X2\le q(\mathbf{v})<X.
$$
For every such $\mathbf{v}$, the word $\mathbf{u}\mathbf{v}(2)$ is canonical and begins with $\mathbf{u}$.  By
\eqref{eq:continuant-concatenation} and the recurrence for appending the final
digit $2$,
$$
 2q(\mathbf{u})q(\mathbf{v})\le q\bigl(\mathbf{u}\mathbf{v}(2)\bigr)\le6q(\mathbf{u})q(\mathbf{v}).
$$
Consequently
$$
 T\le q\bigl(\mathbf{u}\mathbf{v}(2)\bigr)<6T.
$$
Distinct words $\mathbf{v}$ give distinct canonical fractions, which proves the lower
bound.  The last comparison in \eqref{eq:fixed-prefix-count} follows from
Lemma~\ref{lem:cylapprox}.
\end{proof}

\section{Product-set estimates modulo composite moduli}\label{sec:hitting}

Put $e_q(t)=e^{2\pi i t/q}$.  By an interval in $\Z/q\Z$ of length $h$ we
mean a set of $h$ consecutive residue classes.  The estimates in this section are mostly classical and in some form appeared in Schleischitz \cite{SchleischitzULC}, but we get the composite moduli statement, whereas \cite{SchleischitzULC} used the prime moduli case.  The first is Vinogradov's bilinear bound, see also
\cite[Ch.~VI]{Vinogradov}.

\begin{lemma}\label{lem:bilinear}
Let $q\ge2$, let $P,M\subseteq(\Z/q\Z)^\times$, and set
$$
 S_\ell(P,M)=\sum_{p\in P}\sum_{m\in M}e_q(\ell pm).
$$
Then
$$
 |S_\ell(P,M)|\le \sqrt{q(\ell,q)|P||M|}.
$$
\end{lemma}

\begin{proof}
Cauchy--Schwarz gives
$$
 |S_\ell(P,M)|^2
 \le |P|\sum_{p\bmod q}
 \left|\sum_{m\in M}e_q(\ell pm)\right|^2.
$$
After expanding the square, orthogonality in $p$ shows that the inner sum is
nonzero only when
$$
 q\mid \ell(m-m').
$$
If $g=(\ell,q)$, this is equivalent to
$m\equiv m'\pmod{q/g}$.  For each $m\in M$ there are at most $g$ possible
residue classes $m'\pmod q$ satisfying this congruence.  The right-hand side
is therefore at most $|P|qg|M|$.
\end{proof}

Fourier inversion and Lemma~\ref{lem:bilinear} give the required interval count.

\begin{lemma}\label{lem:bilinear-count}
Let $q\ge2$, let $P,M\subseteq(\Z/q\Z)^\times$, and let
$H\subset\Z/q\Z$ be an interval of length $h$.  For every $\eps>0$,
\begin{equation}\label{eq:bilinear-count}
 \#\{(p,m)\in P\times M:pm\in H\}
 =\frac{h}{q}|P||M|
 +O_\eps\left(q^{1/2+\eps}(|P||M|)^{1/2}\right).
\end{equation}
\end{lemma}

\begin{proof}
For $f:\Z/q\Z\to\mathbb C$, use the normalized Fourier transform
$$
 \widehat f(\ell)=\frac1q\sum_{x\bmod q}f(x)e_q(-\ell x),
 \qquad
 f(x)=\sum_{\ell\bmod q}\widehat f(\ell)e_q(\ell x).
$$
Consequently,
$$
 \#\{(p,m)\in P\times M:pm\in H\}
 =\sum_{\ell\bmod q}\widehat{1_H}(\ell)S_\ell(P,M).
$$
The term $\ell=0$ is $h|P||M|/q$.  A geometric-series estimate gives, for
$1\le\ell<q$,
$$
 |\widehat{1_H}(\ell)|
 \ll \min\left\{\frac hq,
 \frac1{\min(\ell,q-\ell)}\right\}.
$$
Using Lemma~\ref{lem:bilinear}, pairing $\ell$ with $q-\ell$, and grouping
according to $d=(\ell,q)$, the contribution of the nonzero frequencies is at
most
\begin{align*}
 (q|P||M|)^{1/2}
 \sum_{1\le\ell\le q/2}\frac{(\ell,q)^{1/2}}{\ell}
 &\ll (q|P||M|)^{1/2}
 \sum_{d\mid q}d^{-1/2}\sum_{m\le q/d}\frac1m\\
 &\ll (q|P||M|)^{1/2}\tau(q)\log q.
\end{align*}
The elementary divisor bound $\tau(q)\log q\ll_\eps q^\eps$ completes the
proof.
\end{proof}

What we need is that only a small set of multipliers can send
every element of $M$ outside a fixed interval.

\begin{corollary}\label{lem:interval-hitting}
Let $q\ge2$, let $\varnothing\ne M\subseteq(\Z/q\Z)^\times$, and let
$H\subset\Z/q\Z$ be an interval of length $|H|\ge dq$, where $d>0$ is fixed.
Define
$$
 E=\{e\in(\Z/q\Z)^\times:eM\cap H=\varnothing\}.
$$
Then, for every $\eps>0$,
\begin{equation}\label{eq:Ebound}
 |E|\ll_{d,\eps}\frac{q^{1+\eps}}{|M|}.
\end{equation}
\end{corollary}

\begin{proof}
Apply Lemma~\ref{lem:bilinear-count} with $P=E$.  The count on the left of
\eqref{eq:bilinear-count} is zero, while $|H|/q\ge d$.  Therefore
$$
 |E||M|\ll_{d,\eps}
 q^{1/2+\eps}(|E||M|)^{1/2}.
$$
Squaring and replacing $2\eps$ by $\eps$ gives \eqref{eq:Ebound}.
\end{proof}

\section{Getting a lower bound in ULC}\label{sec:rational-estimate}

The following elementary lemma gives the lower bounds used in the construction.

\begin{lemma}\label{lem:finite-pair}
Fix $A\ge2$.  Suppose $q\ge2$, $p\in\Z$ with $(p,q)=1$,
$a\in M_A(q)$, and $r\in\Z$ satisfies
$$
 r\equiv pa\pmod q.
$$
Put $\gamma_A=1/(A+2)$.  Then, for every $1\le n<q$,
\begin{equation}\label{eq:finitepair}
 \left\|\frac{np}{q}\right\|
 \left\|\frac{nr}{q}\right\|\ge\frac{\gamma_A}{q},
\end{equation}
and
\begin{equation}\label{eq:finitepair-individual}
 \left\|\frac{np}{q}\right\|,
 \left\|\frac{nr}{q}\right\|\ge\frac{2\gamma_A}{q}.
\end{equation}
\end{lemma}

\begin{proof}
Let $m\in\{1,\ldots,q-1\}$ represent $np\pmod q$.  Since
$r\equiv pa\pmod q$,
$$
 \left\|\frac{np}{q}\right\|
 \left\|\frac{nr}{q}\right\|
 =\left\|\frac mq\right\|
  \left\|\frac{ma}{q}\right\|.
$$
Replacing $m$ by $q-m$ does not change either distance, so we may assume
$m\le q/2$.  Lemma~\ref{lem:finite-badness} now gives
$$
 \left\|\frac mq\right\|
 \left\|\frac{ma}{q}\right\|
 =\frac1q m\left\|\frac{ma}{q}\right\|
 \ge\frac1{(A+2)q}.
$$
This is \eqref{eq:finitepair}.  Since either factor in the product is at most
$1/2$, \eqref{eq:finitepair} also implies
\eqref{eq:finitepair-individual}.
\end{proof}

The next lemma shows that the preceding bounds still hold under $O(q^{-2})$ perturbations.

\begin{lemma}\label{lem:limit-pair}
Fix $A\ge2$ and $K\ge1$.  Suppose that for infinitely many integers
$q\to\infty$ there are integers $p,a,r$ such that
$$
 1\le p,r<q,\qquad (p,q)=1,\qquad a\in M_A(q),
 \qquad r\equiv pa\pmod q,
$$
and
\begin{equation}\label{eq:quadratic-approximation-pair}
 \left|\xi-\frac pq\right|\le\frac K{q^2},
 \qquad
 \left|\zeta-\frac rq\right|\le\frac K{q^2}.
\end{equation}
Then $(\xi,\zeta)\in \mathcal U$.
\end{lemma}

\begin{proof}
Let $\gamma_A=1/(A+2)$ and choose
$$
 0<\lambda\le\min\left\{\frac12,\frac{\gamma_A}{2K}\right\}.
$$
For one of the moduli in the hypothesis, set $Q=\lfloor\lambda q\rfloor$.
For all sufficiently large $q$ and every $1\le n\le Q$, Lemma~\ref{lem:finite-pair} gives
$$
 \left\|\frac{np}{q}\right\|,
 \left\|\frac{nr}{q}\right\|\ge\frac{2\gamma_A}{q},
 \qquad
 \left\|\frac{np}{q}\right\|
 \left\|\frac{nr}{q}\right\|\ge\frac{\gamma_A}{q}.
$$
The function $x\mapsto\|x\|$ is $1$-Lipschitz.  Hence
\eqref{eq:quadratic-approximation-pair} changes either factor by at most
$$
 \frac{Kn}{q^2}\le\frac{K\lambda}{q}
 \le\frac{\gamma_A}{2q},
$$
which is at most one quarter of its rational lower bound.  Therefore
$$
 \|n\xi\|\,\|n\zeta\|
 \ge\frac9{16}
 \left\|\frac{np}{q}\right\|
 \left\|\frac{nr}{q}\right\|
 \ge\frac{9\gamma_A}{16q}
 \qquad(1\le n\le Q).
$$
It follows that
$$
 Q\min_{1\le n\le Q}\|n\xi\|\,\|n\zeta\|
 \ge\frac{9\gamma_A}{16}\frac Qq.
$$
Along the chosen sequence, $Q/q\to\lambda$, so the limsup is positive.
\end{proof}

\section{The pair-dimensional construction}\label{sec:pair-dimension}

This section is devoted to a proof of Theorem~\ref{thm:pair-dim}.  Throughout this section, write
$\delta=\delta_A$. For denominators $q\asymp T$, Hensley's theorem gives
about $T^{2\delta}$ choices for the first coordinate, and on each sufficiently
rich modulus we retain $T^\beta$ multipliers, where $\beta$ is some parameter specified below.  The resulting rectangles have
diameter $O(T^{-2})$ and their number is of order
$$
 T^{2\delta+\beta}=T^{2s},\qquad s=\delta+\frac\beta2.
$$


Once $A$ and $\beta$ are fixed, implicit constants may depend on them, but not
on the prefix, interval, $T$, or ball.

\medskip
\noindent\emph{Admissible rectangles.}
Let $\mathbf{u}$ be an admissible word and let $J\subset(0,1)$ be a nondegenerate
interval.  An admissible $T$-rectangle over $(\mathbf{u},J)$ is a set
\begin{equation}\label{eq:pair-rectangle-form}
 F_A(\mathbf{w}(p/q))\times
 \left[\frac rq-\frac{K_A}{q^2},\frac rq+\frac{K_A}{q^2}\right]
 \subset F_A(\mathbf{u})\times J,
\end{equation}
where $T\le q<6T$, $p\in M_A(q;F_A(\mathbf{u}))$, $b\in M_A(q)$,
and $r\in\{1,\ldots,q-1\}$ is the least positive residue satisfying
$r\equiv pb\pmod q$.

\begin{lemma}\label{lem:pair-rectangles}
Assume
\begin{equation}\label{eq:beta-range}
 \max\left\{2-2\delta,\,-1+\sqrt{5-2\delta}\right\}<\beta<2\delta-1,
\end{equation}
and put
\begin{equation}\label{eq:s-a-def}
 s=\delta+\frac\beta2,\qquad \alpha=\delta+1.
\end{equation}
There is $\chi>0$ such that, for every admissible word $\mathbf{u}$ and every
nondegenerate interval $J\subset(0,1)$, there is $T_0=T_0(\mathbf{u},J)$ for
which every $T\ge T_0$ admits a finite family $\mathcal R_T(\mathbf{u},J)$ of
admissible $T$-rectangles satisfying
\begin{equation}\label{eq:pair-rectangle-lower}
 \#\mathcal R_T(\mathbf{u},J)
 \gg (\diam F_A(\mathbf{u}))^\delta |J|T^{2s}.
\end{equation}
The rectangles are pairwise separated by $\gg T^{-2}$.  Moreover, every ball
$B(z,\rho)\subset\R^2$ satisfies
\begin{align}
 \#\{Q\in\mathcal R_T(\mathbf{u},J):Q\cap B(z,\rho)\ne\varnothing\}
 &\ll (1+\rho T^2)^s &&(\rho\ge T^{-2}),
 \label{eq:pair-local-s}\\
 \#\{Q\in\mathcal R_T(\mathbf{u},J):Q\cap B(z,\rho)\ne\varnothing\}
 &\ll \rho^\alpha T^{2s} &&(\rho\ge T^{-\chi}).
 \label{eq:pair-local-a}
\end{align}
\end{lemma}

\begin{proof}
Choose $\eps>0$ so small that
\begin{equation}\label{eq:eps-choice-pair}
 s>1+\eps,
 \qquad
 \left(1-\frac\beta2\right)
 \frac{2+\beta/2+\eps}{1+\delta/2}<1.
\end{equation}
The first inequality is possible because $s>1$, which is equivalent to
$\beta>2-2\delta$. The second inequality is needed for some local estimate in Step~3 below. 

Fix
\begin{equation}\label{eq:chi-choice}
 0<\chi<\min\left\{1,
 \frac{s-1-\eps}{1+\delta/2}\right\}.
\end{equation}

\medskip
\noindent\emph{Step 1: construction and total count.}
Call $q\in[T,6T)$ \emph{rich} if $|M_A(q)|\ge T^\beta$.  For each
rich $q$, choose a subset $M_q'\subset M_A(q)$ with
$|M_q'|\asymp T^\beta$. For instance, for large $T$ we may require
$T^\beta/2\le |M_q'|\le T^\beta$.  Consider the total prefix count
$$
 S(T;\mathbf{u})=\sum_{T\le q<6T}|M_A(q;F_A(\mathbf{u}))|,
$$
which counts pairs $(q,p)$ with $T\le q<6T$ and $p\in M_A(q;F_A(\mathbf{u}))$.
If $q$ is not rich, then
$|M_A(q;F_A(\mathbf{u}))|\le |M_A(q)|<T^\beta$.  Since there are $O(T)$ possible
moduli $q\in[T,6T)$, the total contribution of the non-rich moduli to
$S(T;\mathbf{u})$ is at most
$$
 \sum_{\substack{T\le q<6T\\ q\ \mathrm{not\ rich}}}
 |M_A(q;F_A(\mathbf{u}))|\ll T^{1+\beta}.
$$
The inequality $1+\beta<2\delta$ and Lemma~\ref{lem:fixed-prefix-count} show that
$S(T;\mathbf{u})\gg (\diam F_A(\mathbf{u}))^\delta T^{2\delta}$.  Hence the rich moduli alone
already contribute
\begin{equation}\label{eq:high-prefix-mass}
 \sum_{\substack{T\le q<6T\\q\ \mathrm{rich}}}
 |M_A(q;F_A(\mathbf{u}))|
 \gg (\diam F_A(\mathbf{u}))^\delta T^{2\delta}.
\end{equation}

Choose an interval $J'\Subset J$ with $|J'|\ge |J|/3$.  For sufficiently
large $T$, every interval of radius $K_Aq^{-2}$ centered at a point
$r/q\in J'$ is contained in $J$.  Let
$$
 H_q=\{r\bmod q:1\le r<q,\ r/q\in J'\}.
$$
Then $|H_q|\gg |J|q$.  For a rich modulus $q$, let
$$
 N_q:=\#\{(p,b)\in M_A(q;F_A(\mathbf{u}))\times M_q': pb\bmod q\in H_q\}.
$$
For every pair counted by $N_q$, let $r\in\{1,\ldots,q-1\}$ be the least
positive residue satisfying $r\equiv pb\pmod q$.  Since $r/q\in J'$, the corresponding set
$$
 Q(q,p,b):=F_A(\mathbf{w}(p/q))\times
 \left[\frac rq-\frac{K_A}{q^2},\frac rq+\frac{K_A}{q^2}\right]
$$
is an admissible rectangle contained in $F_A(\mathbf{u})\times J$.  Let
$\widetilde{\mathcal R}_T(\mathbf{u},J)$ be the family of all rectangles $Q(q,p,b)$
obtained in this way.  Applying Lemma~\ref{lem:bilinear-count} with
$$
 P=M_A(q;F_A(\mathbf{u})),\qquad M=M_q',\qquad H=H_q
$$
therefore gives
$$
 N_q=\frac{|H_q|}{q}\,|M_A(q;F_A(\mathbf{u}))|\,|M_q'|
 +O_\eps\!\left(q^{1/2+\eps}|M_A(q;F_A(\mathbf{u}))|^{1/2}T^{\beta/2}\right).
$$
Since $|H_q|\gg |J|q$ and $|M_q'|\asymp T^\beta$, the main term for a
fixed rich modulus is
$$
 \gg |J|T^\beta|M_A(q;F_A(\mathbf{u}))|.
$$
Summing this lower bound over all rich moduli and using
\eqref{eq:high-prefix-mass}, we obtain a total main-term contribution
$$
 \gg |J|T^\beta\sum_{\substack{T\le q<6T\\ q\ \mathrm{rich}}}|M_A(q;F_A(\mathbf{u}))|
 \gg (\diam F_A(\mathbf{u}))^\delta |J|T^{2\delta+\beta}
 =(\diam F_A(\mathbf{u}))^\delta |J|T^{2s}.
$$
The sum of the error terms is, by Cauchy--Schwarz and the upper bound in
Lemma~\ref{lem:fixed-prefix-count},
\begin{align*}
 T^{1/2+\beta/2+\eps}
 \sum_{T\le q<6T}|M_A(q;F_A(\mathbf{u}))|^{1/2}
 &\ll T^{1+\beta/2+\eps}
 \left(\sum_{T\le q<6T}|M_A(q;F_A(\mathbf{u}))|\right)^{1/2}\\
 &\ll T^{\delta+1+\beta/2+\eps}.
\end{align*}
Since $s>1+\eps$, the error is $o(T^{2s})$.  For each fixed
$(\mathbf{u},J)$, the factor $(\diam F_A(\mathbf{u}))^\delta|J|$ in the
main term is positive, so $T_0(\mathbf{u},J)$ may be enlarged until the error
is at most half of that main term.  Thus the lower-bound constant below is
independent of $\mathbf{u}$ and $J$, although the threshold $T_0(\mathbf{u},J)$
may depend on them:
$$
 \sum_{\substack{T\le q<6T\\q\ \mathrm{rich}}}N_q
 \gg (\diam F_A(\mathbf{u}))^\delta |J|T^{2s}.
$$

The map from counted triples $(q,p,b)$ to the associated points $(p/q,r/q)$ is injective.
Suppose first that $p/q\ne p'/q'$.  Both fractions are reduced and have
denominators in $[T,6T)$, so
$$
 \left|\frac pq-\frac{p'}{q'}\right|\ge\frac1{qq'}\ge\frac1{36T^2}.
$$
If $p/q=p'/q'$, reducedness gives $p=p'$ and $q=q'$.  Two distinct counted
triples must then have $b\ne b'$, and multiplication by $p$ is invertible modulo
$q$, so the corresponding residues $r\equiv pb\pmod q$ and
$r'\equiv pb'\pmod q$ are distinct.  Their second coordinates differ by
at least $q^{-1}$.  Thus the map $(q,p,b)\mapsto(p/q,r/q)$ is injective.  Moreover, distinct
canonical fractions have distinct canonical words and hence distinct
first-coordinate cylinders.  Therefore distinct counted triples yield distinct
rectangles, and their associated points are $\gg T^{-2}$-separated.
In particular,
$$
 \#\widetilde{\mathcal R}_T(\mathbf{u},J)
 =\sum_{\substack{T\le q<6T\\q\ \mathrm{rich}}}N_q
 \gg (\diam F_A(\mathbf{u}))^\delta |J|T^{2s}.
$$
We also have
\begin{equation}\label{eq:pair-total-upper}
 \#\widetilde{\mathcal R}_T(\mathbf{u},J)
 \le T^\beta\sum_{T\le q<6T}|M_A(q;F_A(\mathbf{u}))|
 \ll T^{2s}.
\end{equation}

For $Q=Q(q,p,b)\in\widetilde{\mathcal R}_T(\mathbf{u},J)$, write
$c(Q)=(p/q,r/q)$.  The preceding argument gives
$|c(Q)-c(Q')|\ge (36T^2)^{-1}$ for distinct associated points, while every
rectangle has diameter $O_A(T^{-2})$.  Choose a maximal subfamily whose
associated points are separated by $C_AT^{-2}$, with $C_A$ larger than twice
the rectangle-diameter constant.  The $(36T^2)^{-1}$ separation shows that
each selected point excludes only $O_A(1)$ members, so the subfamily has
cardinality comparable to $\widetilde{\mathcal R}_T(\mathbf{u},J)$ and its
rectangles are mutually $\gg_A T^{-2}$-separated.  Thus
\eqref{eq:pair-rectangle-lower} is preserved.

\medskip
\noindent\emph{Step 2: local counting estimates.}
It is enough to estimate the larger family
$\widetilde{\mathcal R}_T(\mathbf{u},J)$, since $\mathcal R_T(\mathbf{u},J)$ is a subfamily.
Fix a ball
$B(z,\rho)$ with $\rho\ge T^{-2}$, let $N(\rho)$ be the number of rectangles
meeting it, and put
$$
 X=1+\rho T^2.
$$
If a rectangle meets the ball, its associated point $(p/q,r/q)$ lies in a product $I_1\times I_2$
of intervals of length $O_A(\rho)$.  Define
$$
 P_q=\{p\in M_A(q;F_A(\mathbf{u})):p/q\in I_1\}.
$$
Lemma~\ref{lem:local-rational-upper} gives
\begin{equation}\label{eq:local-p-mass}
 \sum_{T\le q<6T}|P_q|\ll X^\delta.
\end{equation}

For fixed $q$ and $p\in P_q$, the map $b\mapsto pb\pmod q$ is injective, and
$I_2$ contains $O(1+\rho q)$ residue points of the form $r/q$.  Hence
\begin{equation}\label{eq:pair-crude-local}
 N(\rho)\ll (1+\rho T)X^\delta.
\end{equation}

For a second estimate, let
$H_{q,2}=\{r\bmod q:0\le r<q,\ r/q\in I_2\}$.  Its length is $O(1+\rho q)$.
For rich $q$, Lemma~\ref{lem:bilinear-count}, applied with $P=P_q$,
$M=M_q'$, and $H=H_{q,2}$, gives after summing over the rich moduli
\begin{align}
 N(\rho)
 &\ll (\rho+T^{-1})T^\beta\sum_{\substack{T\le q<6T\\q\ \mathrm{rich}}}|P_q|
 +T^{1/2+\beta/2+\eps}\sum_{\substack{T\le q<6T\\q\ \mathrm{rich}}}|P_q|^{1/2}\notag\\
 &\ll (\rho+T^{-1})T^\beta X^\delta
 +T^{1+\beta/2+\eps}X^{\delta/2}.
 \label{eq:pair-bilinear-local}
\end{align}
In the last line we used \eqref{eq:local-p-mass} and Cauchy--Schwarz.

\medskip
\noindent\emph{Step 3: proof of \eqref{eq:pair-local-s}.}
If $X\le2T$, then $1+\rho T\ll1$, and \eqref{eq:pair-crude-local} gives
$N(\rho)\ll X^\delta\le X^s$.  If $X\ge T^2$, \eqref{eq:pair-total-upper} gives
$N(\rho)\ll T^{2s}\le X^s$.

It remains to consider $2T<X<T^2$.  In this range
$\rho+T^{-1}\ll X/T^2$, so the main term in
\eqref{eq:pair-bilinear-local} is at most
$$
 T^{\beta-2}X^{\delta+1}
 =X^s\left(\frac{X}{T^2}\right)^{1-\beta/2}
 \le X^s.
$$
For the other contribution, use the better of the crude bound and the error
term in \eqref{eq:pair-bilinear-local}.  After division by $X^s$, these two
bounds become
$$
 f_1(X)=T^{-1}X^{1-\beta/2},
 \qquad
 f_2(X)=T^{1+\beta/2+\eps}X^{-(\delta+\beta)/2}.
$$
The function $f_1$ increases and $f_2$ decreases.  At the left endpoint,
$$
 f_1(2T)\ll T^{-\beta/2}=o(1),
 \qquad
 f_2(2T)\asymp T^{1+\eps-\delta/2}\to\infty,
$$
whereas at the right endpoint,
$$
 f_1(T^2)=T^{1-\beta}\to\infty,
 \qquad
 f_2(T^2)=T^{1+\eps-s}=o(1).
$$
Here $0<\beta<1$ and $s>1+\eps$.  Thus the graphs cross exactly once in the
intermediate range.  At the crossing,
$$
 X=T^{(2+\beta/2+\eps)/(1+\delta/2)},
$$
and their common value is
$$
 T^{-1+(1-\beta/2)(2+\beta/2+\eps)/(1+\delta/2)}.
$$
Thus, in order that the better of the two bounds remain uniformly bounded
(and in fact be $o(1)$ at the crossing), we need
$$
 \left(1-\frac\beta2\right)
 \frac{2+\beta/2+\eps}{1+\delta/2}<1,
$$
which is the second condition in \eqref{eq:eps-choice-pair}.  Setting
$\eps=0$ and expanding gives
$$
 \beta^2+2\beta+2\delta-4>0.
$$
Since $\beta>0$, this is equivalent to
$$
 \beta>-1+\sqrt{5-2\delta}.
$$
This is the origin of the extra lower bound on $\beta$ in \eqref{eq:beta-range}.
Since $\min\{f_1,f_2\}$ increases up to the crossing and decreases afterwards,
it is bounded throughout the intermediate range.  This proves
\eqref{eq:pair-local-s}.

\medskip
\noindent\emph{Step 4: proof of \eqref{eq:pair-local-a}.}
For $\rho\ge1$ the claim follows from
the total bound \eqref{eq:pair-total-upper}, since $\rho^\alpha\ge1$.
Thus assume $T^{-\chi}\le\rho\le1$.  Since $\chi<1$, we have
$\rho\gg T^{-1}$ and $X\asymp\rho T^2$.  The main term in
\eqref{eq:pair-bilinear-local} is then
$$
 \ll \rho T^\beta(\rho T^2)^\delta
 =\rho^\alpha T^{2s}.
$$
The ratio of the error term to $\rho^\alpha T^{2s}$ is
$$
 T^{1-s+\eps}\rho^{-1-\delta/2}
 \le T^{1-s+\eps+\chi(1+\delta/2)}\le1
$$
by \eqref{eq:chi-choice}.  This completes the proof.
\end{proof}

\begin{lemma}\label{lem:mass-distribution}
Let $0<t<s\le\alpha$, let $R_k\downarrow0$, and let
$R_k\le\eta_k\le R_{k-1}$.  Suppose $\mathcal E_0=\{P_0\}$ and, for
$k\ge1$, $\mathcal E_k$ is a finite family of nonempty compact subsets of
$\R^d$.  Assume that every $Q\in\mathcal E_k$ is contained in a unique parent
$P\in\mathcal E_{k-1}$, and set
$$
 K=\bigcap_{k\ge0}\bigcup_{Q\in\mathcal E_k}Q.
$$
Suppose that the following estimates hold with constants independent of $k$:
\begin{enumerate}[label=\textup{(\roman*)}]
 \item every $Q\in\mathcal E_k$ has diameter $O(R_k)$, and distinct
 level-$k$ sets are separated by $\gg R_k$;
 \item every level-$(k-1)$ parent has at least
 \begin{equation}\label{eq:two-regime-child-lower}
  \gg R_{k-1}^\alpha R_k^{-s}
 \end{equation}
 children;
 \item for every level-$(k-1)$ parent $P$ and every ball $B(x,r)$, the
 number of children of $P$ meeting the ball is
 \begin{align}
  &\ll (r/R_k)^s &&(r\ge R_k),\label{eq:two-regime-local-s}\\
  &\ll r^\alpha R_k^{-s} &&(r\ge\eta_k).
  \label{eq:two-regime-local-a}
 \end{align}
\end{enumerate}
There is a constant $c_0>0$, depending only on the constants in
\textup{(i)}--\textup{(iii)}, such that $\dimH K\ge t$ provided that there is
an index $k_0$ for which
\begin{equation}\label{eq:mass-distribution-conditions}
 R_{k-1}^{t-\alpha}R_k^{s-t}\le c_0\quad(k\ge k_0),
 \qquad
 \sup_{k\ge k_0}R_{k-1}^{t-\alpha}\eta_k^{s-t}<\infty.
\end{equation}
\end{lemma}

\begin{proof}
Hypothesis \textup{(ii)} makes $K$ a nonempty compact set.  Define a probability
measure $\nu$ on $K$ by assigning mass $1$ to $P_0$ and dividing the mass of
each parent equally among its children.  The resulting consistent masses on
the refining level partitions define a Borel probability measure on $K$.
Thus, if $N(P)$ is the number of children of $P$ and $Q$ is one of them,
\textup{(ii)} gives
\begin{equation}\label{eq:child-mass-bound}
 \nu(Q)=\frac{\nu(P)}{N(P)}
 \ll \nu(P)R_{k-1}^{-\alpha}R_k^s.
\end{equation}

We next prove
\begin{equation}\label{eq:two-regime-cylinder-mass}
 \nu(Q)\ll R_k^t\qquad(Q\in\mathcal E_k)
\end{equation}
by induction.  Let $k_0$ be as in \eqref{eq:mass-distribution-conditions}.  Since there are only finitely many sets in
the levels $0,\ldots,k_0$, choose $C$ so that
$\nu(Q)\le C R_j^t$ for every $Q\in\mathcal E_j$ with $j\le k_0$.
Let $c_b>0$ be the constant in \textup{(ii)}, so that
$N(P)\ge c_bR_{k-1}^\alpha R_k^{-s}$.  If
$\nu(P)\le C R_{k-1}^t$, then
$$
 \nu(Q)
 \le c_b^{-1}C R_k^t
 \bigl(R_{k-1}^{t-\alpha}R_k^{s-t}\bigr).
$$
Choose $c_0\le c_b$.  The first condition in
\eqref{eq:mass-distribution-conditions} then gives
$\nu(Q)\le C R_k^t$, closing the induction and proving
\eqref{eq:two-regime-cylinder-mass}.

Fix a sufficiently small $r$ and choose $k>k_0$ with
$R_k\le r<R_{k-1}$.  By \textup{(i)}, a ball of radius $r$ meets only $O(1)$
level-$(k-1)$ parents, with a uniform constant.  Fix one such parent $P$.
By \eqref{eq:child-mass-bound}, the mass of $B(x,r)\cap P$ is bounded by the
number of children of $P$ meeting $B(x,r)$ multiplied by
$\ll \nu(P)R_{k-1}^{-\alpha}R_k^s$.

If $r\ge\eta_k$, the second local count gives
$$
 \nu(B(x,r)\cap P)
 \ll \nu(P)R_{k-1}^{-\alpha}r^\alpha
 \ll R_{k-1}^{t-\alpha}r^\alpha
 =r^t\left(\frac r{R_{k-1}}\right)^{\alpha-t}
 \le r^t.
$$
If $r<\eta_k$, the first local count gives
$$
 \nu(B(x,r)\cap P)
 \ll R_{k-1}^{t-\alpha}r^s
 =r^t R_{k-1}^{t-\alpha}r^{s-t}
 \le r^t R_{k-1}^{t-\alpha}\eta_k^{s-t}
 \ll r^t
$$
by the second condition in \eqref{eq:mass-distribution-conditions}.  Summing over the
$O(1)$ possible parents, we obtain
\begin{equation}\label{eq:frostman-ball-bound}
 \nu(B(x,r))\ll r^t
\end{equation}
for all sufficiently small $r$.

The mass distribution principle applied to
\eqref{eq:frostman-ball-bound} now gives $\dimH K\ge t$.
\end{proof}

\begin{proof}[Proof of Theorem~\ref{thm:pair-dim}]
Let $\delta=\delta_A>(\sqrt{21}-1)/4$.  The interval in
\eqref{eq:beta-range} is nonempty: the inequality
$2-2\delta<2\delta-1$ follows from $\delta>3/4$, while
$$
 -1+\sqrt{5-2\delta}<2\delta-1
$$
is equivalent to $4\delta^2+2\delta-5>0$.  Choose $\beta$ in that
interval and retain the notation
$$
 s=\delta+\frac\beta2,\qquad \alpha=\delta+1.
$$
The range \eqref{eq:beta-range} gives $0<\beta<1$, and hence $s<\alpha$.

Fix $0<t<s$.  We construct a compact set
$K_t\subset \mathcal U\cap(F_A\times\R)$ with $\dimH K_t\ge t$.
Let
$$
 \mathcal E_0=\{P_0\},\qquad P_0=F_A(\mathbf{u}_0)\times J_0,
$$
where $\mathbf{u}_0$ is admissible and $J_0\Subset(0,1)$ is a nondegenerate compact
interval.  Put $T_0=R_0=1$.

If $P=F_A(\mathbf{u})\times J\in\mathcal E_{k-1}$ with $k\ge2$, then
\eqref{eq:pair-rectangle-form}, $T_{k-1}\le q<6T_{k-1}$, and
Lemma~\ref{lem:cylapprox} give
$$
 \diam F_A(\mathbf{u})\asymp R_{k-1},\qquad |J|\asymp R_{k-1}.
$$
Since the implicit constants in
\eqref{eq:pair-rectangle-lower}--\eqref{eq:pair-local-a} are independent of
$(\mathbf{u},J)$, the constants required in \textup{(i)}--\textup{(iii)} of
Lemma~\ref{lem:mass-distribution} are independent of $k$ and of $P$.  Decrease
the constant in the lower bound \textup{(ii)}, if necessary, so that it also
covers $P_0$, and let $c_0$ be the constant supplied by that lemma.  Suppose that the
finite family $\mathcal E_{k-1}$ has been constructed.  Choose
$T_k\ge2T_{k-1}$ so large that Lemma~\ref{lem:pair-rectangles} applies to every
parent in $\mathcal E_{k-1}$ and, with
$$
 R_k=T_k^{-2},\qquad \eta_k=T_k^{-\chi},
$$
we have
\begin{equation}\label{eq:pair-parameter-choice}
 \eta_k\le R_{k-1},\qquad
 R_{k-1}^{t-\alpha}R_k^{s-t}\le c_0,
 \qquad
 R_{k-1}^{t-\alpha}\eta_k^{s-t}\le1.
\end{equation}
This is possible because $\mathcal E_{k-1}$ is finite and both
$R_k^{s-t}$ and $\eta_k^{s-t}$ tend to zero as $T_k\to\infty$.  Notice also
that $R_k\le\eta_k$, since $\chi<1<2$.

For each parent $P=F_A(\mathbf{u})\times J$, take the rectangles in
$\mathcal R_{T_k}(\mathbf{u},J)$ as its children, and let $\mathcal E_k$ be their
union.  Children of the same parent are $\gg R_k$-separated by
Lemma~\ref{lem:pair-rectangles}.  If two children have distinct parents, the induction hypothesis gives
distance $\gg R_{k-1}$ between those parents, hence between the children; this
is $\gg R_k$ because $T_k\ge2T_{k-1}$.  Hence the level-$k$ rectangles are
$\gg R_k$-separated and every child has a unique parent.  Define
\begin{equation}\label{eq:pair-cantor-set}
 K_t=\bigcap_{k\ge0}\bigcup_{P\in\mathcal E_k}P.
\end{equation}
The unions are nonempty nested compact sets, so $K_t$ is nonempty and compact.

We verify the hypotheses of Lemma~\ref{lem:mass-distribution}.  Every
level-$k$ rectangle has diameter $O(R_k)$.  If $P=F_A(\mathbf{u})\times J\in\mathcal E_{k-1}$ with $k\ge2$, then
$\diam F_A(\mathbf{u})\asymp R_{k-1}$ and $|J|\asymp R_{k-1}$, as above.  Therefore
\eqref{eq:pair-rectangle-lower} gives
$$
 \#\{\text{children of }P\}
 \gg R_{k-1}^{\delta+1}R_k^{-s}
 =R_{k-1}^{\alpha}R_k^{-s}.
$$
For $P_0$, the same lower bound holds after the one-time decrease of its
constant made above.  The
two local counting hypotheses follow from
\eqref{eq:pair-local-s}--\eqref{eq:pair-local-a}, because for $r\ge R_k$,
$$
 1+rT_k^2\asymp r/R_k,
 \qquad T_k^{2s}=R_k^{-s}.
$$
Moreover $R_k\le\eta_k$ because $\chi<2$, while
\eqref{eq:pair-parameter-choice} gives $\eta_k\le R_{k-1}$ and both
inequalities in \eqref{eq:mass-distribution-conditions}.  Hence
Lemma~\ref{lem:mass-distribution} yields
$$
 \dimH K_t\ge t.
$$

It remains to identify the points of $K_t$.  Since distinct level-$k$
rectangles are disjoint, each $(\xi,\zeta)\in K_t$ belongs to a unique
$P_k\in\mathcal E_k$; the unique-parent property gives
$P_{k+1}\subset P_k$.  Write $P_k=Q(q_k,p_k,b_k)$ as in
\eqref{eq:pair-rectangle-form}, and let $r_k\in\{1,\ldots,q_k-1\}$ be the
least positive residue satisfying $r_k\equiv p_kb_k\pmod{q_k}$.  Then $T_k\le q_k<6T_k$, and the definition of $P_k$ together with
\eqref{eq:KA-choice} gives
\begin{equation}\label{eq:pair-limit-approx}
 \left|\xi-\frac{p_k}{q_k}\right|\le\frac{K_A}{q_k^2},
 \qquad
 \left|\zeta-\frac{r_k}{q_k}\right|\le\frac{K_A}{q_k^2}.
\end{equation}
The first-coordinate words are nested.  Their denominators tend to infinity,
so their lengths tend to infinity and they determine an infinite bounded-digit
expansion; thus $\xi\in F_A^\infty$.  Lemma~\ref{lem:limit-pair}, applied with $(q,p,a,r)=(q_k,p_k,b_k,r_k)$
and $K=K_A$, gives $(\xi,\zeta)\in \mathcal U$.  Consequently
$$
 K_t\subset \mathcal U\cap(F_A^\infty\times\R).
$$

Since this construction works for every $t<s$,
$$
 \dimH\bigl(\mathcal U\cap(F_A\times\R)\bigr)\ge s
 =\delta+\frac\beta2.
$$
Letting $\beta\uparrow2\delta-1$ gives
$$
 \dimH\bigl(\mathcal U\cap(F_A\times\R)\bigr)
 \ge2\delta_A-\frac12.
$$
Finally, $F_A^\infty\subset\Bad_1$ and $\delta_A\to1$, so taking the supremum
over $A$ gives
$$
 \dimH\bigl(\mathcal U\cap(\Bad_1\times\R)\bigr)\ge\frac32.
$$
\end{proof}

\begin{proof}[Proof of Corollary~\ref{cor:main}]
Choose $A$ with $\delta_A>(\sqrt{21}-1)/4$.  Theorem~\ref{thm:pair-dim} gives
$$
 \dimH\bigl(\mathcal U\cap(F_A\times\R)\bigr)>0,
$$
so the set is nonempty.  Its first coordinate cannot be rational: if
$\xi=u/v$, then for every $Q\ge v$ the choice $n=v$ gives $\|n\xi\|=0$.
Thus the first coordinate lies in $F_A^\infty\subset\Bad_1$, which proves the
corollary.
\end{proof}

\section{Full dimension in first coordinate}\label{sec:first-coordinate}

We prove Theorem~\ref{thm:first-dim} separately because the second coordinate
is used only to choose intervals that make the pairs belong to $\mathcal U$;
it does not enter the dimension count.

\begin{lemma}\label{lem:many-good-p}
Fix $A$ with $\delta=\delta_A>3/4$.  Let $\mathbf{u}$ be an admissible word and let
$J\subset(0,1)$ be a nondegenerate interval.  There is
$T_0=T_0(\mathbf{u},J)$ such that, for every $T\ge T_0$,
\begin{equation}\label{eq:many-good-uniform}
 \#\left\{(p,q):
 \begin{array}{l}
 T\le q<6T,\ p\in M_A(q;F_A(\mathbf{u})),\\
 \text{there exist }b\in M_A(q)\text{ and }1\le r<q\\
 \text{ with }r\equiv pb\pmod q\text{ and }r/q\in J
 \end{array}\right\}
 \gg (\diam F_A(\mathbf{u}))^\delta T^{2\delta}.
\end{equation}
There is a subcollection containing a fixed positive proportion of these pairs
such that the sets
$F_A(\mathbf{w}(p/q))$ are mutually $\gg T^{-2}$-separated.  The implicit constants are independent of $\mathbf{u}$ and $J$.
\end{lemma}

\begin{proof}
Choose
$$
 2-2\delta<\beta<2\delta-1,
$$
which is possible because $\delta>3/4$.

Lemma~\ref{lem:fixed-prefix-count} gives
\begin{equation}\label{eq:prefix-lower-for-good}
 \sum_{T\le q<6T}|M_A(q;F_A(\mathbf{u}))|
 \gg (\diam F_A(\mathbf{u}))^\delta T^{2\delta}.
\end{equation}
Call a modulus $q$ \emph{rich} if $|M_A(q)|\ge T^\beta$.  The non-rich
moduli contribute at most $T^{1+\beta}$ to the sum in
\eqref{eq:prefix-lower-for-good}.

For each rich modulus, define
$$
 H_q=\{r\bmod q:1\le r<q,\ r/q\in J\},
 \qquad
 E_q=\{e\in(\Z/q\Z)^\times:eM_A(q)\cap H_q=\varnothing\}.
$$
The set $E_q$ consists precisely of the first-coordinate residues that cannot
be paired with any bounded-digit multiplier to hit $J$.  Since
$|H_q|\gg_J q$ and $|M_A(q)|\ge T^\beta$, Corollary~\ref{lem:interval-hitting} gives
$$
 |E_q|\ll_{A,J,\eps}T^{1+\eps-\beta}.
$$
Thus the exceptional pairs contribute
$O_{A,J,\eps}(T^{2+\eps-\beta})$.  Choose $\eps>0$ with
$2+\eps-\beta<2\delta$.  Since also $1+\beta<2\delta$, both losses have
strictly smaller powers of $T$ than \eqref{eq:prefix-lower-for-good}.  After
fixing $(\mathbf{u},J)$, the factor $(\diam F_A(\mathbf{u}))^\delta$ in
\eqref{eq:prefix-lower-for-good} and the $J$-dependent implied constant above
are fixed; hence $T_0(\mathbf{u},J)$ may be enlarged so that both losses are
absorbed.  The lower-bound constant in \eqref{eq:many-good-uniform} remains
independent of $\mathbf{u}$ and $J$.

Finally, distinct reduced fractions with denominators below $6T$ are
$(36T^2)^{-1}$-separated, while Lemma~\ref{lem:cylapprox} gives cylinder
diameter $O_A(T^{-2})$.  Choose a maximal subcollection whose rational
endpoints are separated by $C_AT^{-2}$, with $C_A$ larger than twice the
cylinder-diameter constant.  The Farey separation shows that each selected
endpoint accounts for only $O_A(1)$ original endpoints, so this subcollection
contains a fixed positive proportion of the pairs in
\eqref{eq:many-good-uniform}; its cylinders are mutually $\gg_A T^{-2}$-separated.
\end{proof}

\begin{proof}[Proof of Theorem~\ref{thm:first-dim}]
Fix $A$ with $\delta=\delta_A>3/4$.  We first prove
$\dimH\cP_A\ge\delta$.

Fix $0<t<\delta$, a nonempty cylinder $C_0=F_A(\mathbf{w}_0)$, and a compact interval
$J_0\Subset(0,1)$ with nonempty interior.  At each stage we keep pairs $(C,J)$,
where $C$ is a first-coordinate cylinder and $J$ is a closed second-coordinate
interval.  Only the cylinders $C$ enter the dimension estimate.  Let $\mathcal D_0=\{(C_0,J_0)\}$ and put $T_0=R_0=1$.

We shall use the following estimate, uniform in the admissible word
$\mathbf{u}$:
\begin{equation}\label{eq:first-local-count}
 \#\left\{(p,q):
 \begin{array}{l}
 T\le q<6T,\ p\in M_A(q;F_A(\mathbf{u})),\\
 F_A(\mathbf{w}(p/q))\cap B(x,r)\ne\varnothing
 \end{array}\right\}
 \ll (rT^2)^\delta\qquad(r\ge T^{-2}).
\end{equation}
It follows from Lemmas~\ref{lem:cylapprox} and
\ref{lem:local-rational-upper}, and remains valid after restricting the set of
pairs $(p,q)$.

If $C\in\mathcal C_{k-1}$ with $k\ge2$, then
$C=F_A(\mathbf{w}(p/q))$ for some $T_{k-1}\le q<6T_{k-1}$, and
Lemma~\ref{lem:cylapprox} gives $\diam C\asymp R_{k-1}$.  Hence
Lemma~\ref{lem:many-good-p} gives
$$
 \#\{\text{children of }C\}\gg R_{k-1}^\delta R_k^{-\delta}
$$
with a constant independent of $C$, while \eqref{eq:first-local-count} gives
the local estimate in \textup{(iii)} with a constant independent of $C$.
Only the threshold in Lemma~\ref{lem:many-good-p} depends on $(C,J)$.
Decrease the constant in the lower bound \textup{(ii)}, if necessary, so that
it also covers $C_0$, and let $c_0$ be the constant supplied by
Lemma~\ref{lem:mass-distribution}.  Suppose that the finite family
$\mathcal D_{k-1}$ has been constructed.  For each parent $(C,J)$
choose a compact interval $J^*\Subset J$ with nonempty interior.  Since there are only finitely many parents, we may choose
$T_k\ge2T_{k-1}$ so large that all of the following hold simultaneously:
Lemma~\ref{lem:many-good-p} applies to every pair $(C,J^*)$; the radius
$K_AT_k^{-2}$ is smaller than the distance from every $J^*$ to the boundary of
its parent interval $J$; and, with $R_k=T_k^{-2}$,
\begin{equation}\label{eq:first-parameter-choice}
 \left(\frac{R_k}{R_{k-1}}\right)^{\delta-t}\le c_0.
\end{equation}

For each parent, use the separated subfamily supplied by
Lemma~\ref{lem:many-good-p}.  For every pair $(p,q)$ in this subfamily, choose
$b\in M_A(q)$ and $1\le r<q$ such that
$r\equiv pb\pmod q$ and $r/q\in J^*$.  Associate to $(p,q)$ the pair
$$
 C'=F_A(\mathbf{w}(p/q)),
 \qquad
 J'=\left[\frac rq-\frac{K_A}{q^2},
          \frac rq+\frac{K_A}{q^2}\right]\subset J.
$$
Let $\mathcal D_k$ be the family of all such pairs $(C',J')$, and let
$\mathcal C_k$ be their first-coordinate cylinders.  Children of the same
parent are $\gg R_k$-separated by Lemma~\ref{lem:many-good-p}.  Children of
distinct parents lie in level-$(k-1)$ cylinders separated by
$\gg R_{k-1}\gg R_k$, by the induction hypothesis and $T_k\ge2T_{k-1}$.
Thus the level-$k$ cylinders are $\gg R_k$-separated, and every pair in
$\mathcal D_k$ arises from a unique pair in $\mathcal D_{k-1}$.  Define
\begin{equation}\label{eq:first-cantor-set}
 E_t=\bigcap_{k\ge0}\bigcup_{C\in\mathcal C_k}C,
 \qquad \mathcal C_0=\{C_0\}.
\end{equation}
The unions are nested nonempty compact sets, so $E_t$ is compact and nonempty.

We apply Lemma~\ref{lem:mass-distribution} to the families
$\mathcal C_k$.  Every level-$k$ cylinder has diameter $O(R_k)$, and the level-$k$ cylinders are $\gg R_k$-separated.  For $C\in\mathcal C_{k-1}$ with $k\ge2$,
Lemma~\ref{lem:many-good-p} and $\diam C\asymp R_{k-1}$ give
$$
 \#\{\text{children}\}
 \gg R_{k-1}^\delta R_k^{-\delta}.
$$
For $C_0$, the same lower bound holds by the one-time choice of the constant
made above.  The local estimate is \eqref{eq:first-local-count}.  We may therefore
take
$$
 s=\alpha=\delta,\qquad \eta_k=R_k
$$
in Lemma~\ref{lem:mass-distribution}.  With these choices, the two left-hand
sides in \eqref{eq:mass-distribution-conditions} are both
$$
 \left(\frac{R_k}{R_{k-1}}\right)^{\delta-t},
$$
which is at most $c_0$ by \eqref{eq:first-parameter-choice}.  Therefore
$$
 \dimH E_t\ge t.
$$

It remains to show that every $\xi\in E_t$ can be paired with some
$\zeta\in\R$.  Since the level-$k$ cylinders are pairwise disjoint, $\xi$
belongs to a unique $C_k\in\mathcal C_k$.  By construction there is a unique
$J_k$ with $(C_k,J_k)\in\mathcal D_k$, and the construction gives
$C_{k+1}\subset C_k$ and $J_{k+1}\subset J_k$.  Since $\diam J_k\to0$, the intervals $J_k$ determine a unique
$\zeta\in\bigcap_kJ_k$.  Write $p_k,q_k,b_k,r_k$ for the integers chosen in the
definition of the pair $(C_k,J_k)$.  Then \eqref{eq:KA-choice} and the
definition of $J_k$ give
$$
 \left|\xi-\frac{p_k}{q_k}\right|\le\frac{K_A}{q_k^2},
 \qquad
 \left|\zeta-\frac{r_k}{q_k}\right|\le\frac{K_A}{q_k^2},
 \qquad
 r_k\equiv p_kb_k\pmod{q_k}.
$$
The first-coordinate words are nested and $q_k\to\infty$, so their lengths
tend to infinity and $\xi\in F_A^\infty\subset\Bad_1$.  Lemma
\ref{lem:limit-pair} gives $(\xi,\zeta)\in \mathcal U$.  Thus
$E_t\subset\cP_A$.  Letting
$t\uparrow\delta$ yields $\dimH\cP_A\ge\delta$.  The reverse inequality follows
from $\cP_A\subset F_A$ and $\dimH F_A=\delta_A$, so
$$
 \dimH\cP_A=\delta_A.
$$

It remains to prove the local full-dimension statement for $\cP$.  Let
$U\subset\R$ be a nonempty bounded open interval.  Choose $m\in\Z$ such that
$V=(U-m)\cap(0,1)$ contains a nonempty open interval, and choose an irrational
point $x=[0;b_1,b_2,\ldots]$ in that interval.  The full continued-fraction
intervals determined by the prefixes of $x$ shrink to $x$ by
\eqref{eq:inverse-branch}; hence some prefix
$\mathbf{w}=(b_1,\ldots,b_n)$ has its full interval contained in $V$.

Given $0<\eps<1$, choose $A$ so large that
$$
 A\ge\max\{b_1,\ldots,b_n\},
 \qquad
 \delta_A>\max\{3/4,1-\eps\},
$$
and then choose $1-\eps<t<\delta_A$.  Running the preceding construction
inside $F_A(\mathbf{w})$ produces a set
$E_t\subset\cP_A\cap V$ with $\dimH E_t\ge t>1-\eps$.

Finally,
$$
 \|n(\xi+m)\|=\|n\xi\|\qquad(n\in\Z),
$$
so integer translation preserves both bad approximability and membership in
the first-coordinate projection of $\mathcal U$.  Hence
$m+E_t\subset\cP\cap U$, and
$$
 \dimH(\cP\cap U)\ge1-\eps.
$$
Letting $\eps\downarrow0$ gives $\dimH(\cP\cap U)=1$.  In particular,
$\dimH\cP=1$.
\end{proof}

\section*{Acknowledgements}
The author thanks Nikolay Moshchevitin for helpful discussions and Johannes Schleischitz for useful comments on the draft. The majority of the work was done during the Simons Semesters in Będlewo and Warsaw, and the author is very grateful to the organizers for hosting this event.

This work was partially supported by the Simons Foundation grant (award no. SFI-MPS-T-Institutes-00010825) and from State Treasury funds as part of a task commissioned by the Minister of Science and Higher Education under the project “Organization of the Simons Semesters at the Banach Center - New Energies in 2026-2028” (agreement no. MNiSW/2025/DAP/491).


\begin{thebibliography}{99}

\bibitem{BadziahinPollingtonVelani}
D.~Badziahin, A.~Pollington and S.~Velani,
\emph{On a problem in simultaneous Diophantine approximation:
Schmidt's conjecture},
Ann. of Math. (2) \textbf{174} (2011), no.~3, 1837--1883.

\bibitem{BandiFregoliKleinbock}
P.~Bandi, R.~Fregoli and D.~Kleinbock,
\emph{Submanifold-genericity of $\R^d$-actions and uniform multiplicative
Diophantine approximation},
arXiv:2504.02258, 2025.

\bibitem{BourgainKontorovich}
J.~Bourgain and A.~Kontorovich,
\emph{On Zaremba's conjecture},
Ann. of Math. (2) \textbf{180} (2014), no.~1, 137--196.

\bibitem{EinsiedlerKatokLindenstrauss}
M.~Einsiedler, A.~Katok and E.~Lindenstrauss,
\emph{Invariant measures and the set of exceptions to Littlewood's conjecture},
Ann. of Math. (2) \textbf{164} (2006), no.~2, 513--560.

\bibitem{Hensley1989}
D.~Hensley,
\emph{The Hausdorff dimensions of some continued fraction Cantor sets},
J. Number Theory \textbf{33} (1989), no.~2, 182--198.

\bibitem{Hensley1990}
D.~Hensley,
\emph{The distribution of badly approximable rationals and continuants with bounded digits, II},
J. Number Theory \textbf{34} (1990), no.~3, 293--334.

\bibitem{Hensley1992}
D.~Hensley,
\emph{Continued fraction Cantor sets, Hausdorff dimension, and functional analysis},
J. Number Theory \textbf{40} (1992), no.~3, 336--358.

\bibitem{KhinchinCF}
A.~Ya. Khinchin,
\emph{Continued fractions},
University of Chicago Press, Chicago, 1964; reprint, Dover, Mineola, NY, 1997.

\bibitem{MoshchevitinULC}
N.~Moshchevitin,
\emph{A note on uniform version of Littlewood inequality and Fibonacci numbers},
arXiv:2605.26188, 2026.

\bibitem{PollingtonVelaniLittlewood}
A.~D. Pollington and S.~L. Velani,
\emph{On a problem in simultaneous Diophantine approximation:
Littlewood's conjecture},
Acta Math. \textbf{185} (2000), no.~2, 287--306.

\bibitem{SchleischitzULC}
J.~Schleischitz,
\emph{Disproof of the uniform Littlewood conjecture},
arXiv:2603.12611, 2026.

\bibitem{SchleischitzULCing}
J.~Schleischitz,
\emph{On an inhomogeneous uniform Littlewood type problem},
arXiv:2607.08114, 2026.



\bibitem{Vinogradov}
I.~M. Vinogradov,
\emph{An introduction to the theory of numbers},
Pergamon Press, London, 1955.

\end{thebibliography}
\end{document}